\documentclass[twoside,12pt]{article}
\usepackage{amsbsy,latexsym,amsopn,amstext,amsxtra,euscript,amscd}
\usepackage{amssymb,amsfonts,amsmath,amsthm}
\usepackage{graphics,graphicx}

\newcommand*{\N}{4}
\newcommand*{\K}{6}
\def\BT{${\cal BIT}$}
\DeclareMathOperator{\sign}{sign}

\usepackage{tikz}
\usetikzlibrary{arrows,chains,matrix,positioning,scopes}
\usetikzlibrary{positioning,shadows,backgrounds}  
\tikzset{>=stealth',every on chain/.append style={join},
         every join/.style={->}}
\tikzstyle{labeled}=[execute at begin node=$\scriptstyle,
   execute at end node=$]

\theoremstyle{plain}
\newtheorem{theorem}{Theorem}
\newtheorem{corollary}[theorem]{Corollary}
\newtheorem{lemma}[theorem]{Lemma}

\theoremstyle{definition}

\theoremstyle{remark}
\newtheorem{remark}[theorem]{Remark}

\title{\bf On the binomial interpolated triangles}

\author{L\'aszl\'o N\'emeth\footnote{University of Sopron, Institute of Mathematics, Hungary. \textit{nemeth.laszlo@uni-sopron.hu}} }
\date{}

\begin{document}

\maketitle \thispagestyle{empty}

\begin{abstract}
The binomial interpolated transform of a sequence is a generalization of the well-known binomial transform. We examine a Pascal-like triangle, on which a binomial interpolated transform works between the left and right diagonals, focusing on binary recurrences. We give the sums of the elements in rows and in rising diagonals, further we define two special classes of these arithmetical triangles.  \\[1mm]     
The final publication is available at \textit{Journal of Integer Sequences}, Vol. 20, Article 17.7.8 via 
https://cs.uwaterloo.ca/journals/JIS/vol20.html.
\end{abstract}

% Primary 11B37; Secondary 11B39, 11B65, 11B75
% 11B37  Recurrences 
% 11B39  Fibonacci and Lucas numbers and polynomials and generalizations
% 11B65  Binomial coefficients; factorials; $q$-identities 
% 11B75  Other combinatorial number theory

\section{Introduction}

Let us define the sequence $b=\{b_n\}_{n=0}^{\infty}\in\mathbb{R}^{\infty}$ as the binomial transform of the given sequence $a=\{a_n\}_{n=0}^{\infty}\in\mathbb{R}^{\infty}$ by $b_n=\sum_{i=0}^{n}\binom{n}{i} a_i$. This transformation is invertible with formula $a_n=\sum_{i=0}^{n}\binom{n}{i} (-1)^{n-i} b_i$. Several researchers  \cite{Barbero,Bhadouria,Falcon,Pan,Spivey} examined the properties and the generalizations of the binomial transformation. 
One of its generalizations is the so-called binomial interpolated transform \cite{Barbero} given by 
\begin{equation}\label{eq:interpolated_trans}
b_n=\sum_{i=0}^{n}\binom{n}{i} u^iv^{n-i}a_i
\end{equation}
for any non-zero $u,v\in \mathbb{R}$. Bhadouria at al.\ \cite{Bhadouria}  and Falcon and Plaza \cite{Falcon} showed for some (falling and rising) $k$-binomial transform cases, when  $u=1$, $v=k$; $u=k$, $v=1$  or $u=v=k$,  that there are special infinite (Pascal-like) triangles, whose left diagonal (left leg) contains the terms of $a$ and the right one (right leg)  the terms of $b$. 

In our paper, we examine the inner part of this type of triangle in a generalized form focusing exclusively on the binary recurrence sequences $a$. We determine the sums and alternating sums of rows, rising diagonals and central term sequence and give explicit forms for entries of the triangle. Finally, we define and examine two special classes of our arithmetical triangles.

\section{Binomial interpolated triangle}
 
Let an arithmetical triangle (called ``derangement triangle" \cite{Bhadouria}) be defined  by terms $a_{n,k}$ ($k$-th entry in row $n$ -- see Figure~\ref{fig:construction_}) from the sequence $\{a_{n,0}\}_{n=0}^{\infty}$, where
\begin{equation}\label{eq:def_triangle_rule}
 a_{n,k} = u a_{{n},{k-1}}+v a_{{n-1},{k-1}} \quad  (1\leq k\leq n),
\end{equation}  $ u, v \in \mathbb{R}$ and $u v\ne0$. 

We shall show that the right diagonal sequence is a binomial interpolated transform of the left diagonal sequence with parameters $u$ and $v$. Moreover, we shall prove that the converse also holds, with the parameters $-v/u$ and $1/u$.

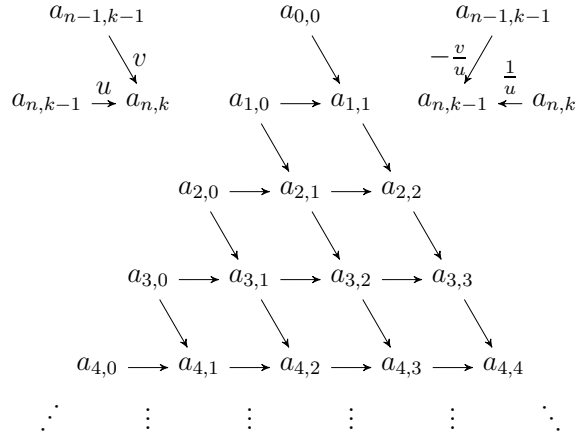
\begin{figure}[!thbp]
	\centering 
	\renewcommand*{\N}{4}
	\renewcommand*{\K}{4}
	
	\scalebox{0.90}{ 		         
		\begin{tikzpicture}[->,xscale=1.5,yscale=1.3, auto,swap]
		\foreach \i in {0,...,\N}
		\foreach \j in {0,...,\i} 
		{\node(a\i \j)  at ({-\i/2 + \j},{-\i}) {$a_{{\i},{\j}}$};}   

		\pgfmathtruncatemacro{\i}{\N +1 };
		\node[rotate=25](a\i 0)  at ({-\i/2 },{0.5-\i}) {\reflectbox{$\ddots$}};
		\foreach \j in {1,...,{\numexpr\N }}
		{
			\node(a\i \j)  at ({-\i/2 +\j},{0.5-\i}) {$\vdots$};
		}
		\pgfmathtruncatemacro{\j}{(\N + 1) }
		\node[rotate=-25] (a\i \j)  at ({-\i/2 +\j},{0.5-\i}) {$\ddots$};	

		\foreach \i in {1,...,\N}
		\foreach \j in {1,...,\i}
		{  \pgfmathtruncatemacro{\x}{(\i -1) };
			\pgfmathtruncatemacro{\y}{\j -1 };
			\path[every node/.style={scale=.7},midway]  (a\i \y) edge node [above]{} (a\i \j) ; 
			\path[every node/.style={scale=.7},midway] (a\x \y) edge node [above]{} (a\i \j) ;
		}		

		\node (c0) at (-2,-0)   {{$a_{{n-1},{k-1}}$}}; 
		\node (c1) at (-2.5,-1)   {{$a_{{n},{k-1}}$}}; 
		\node (c2) at (-1.5,-1)   {{$a_{n,k} $}};
		\path (c0) edge node [right]{$v$} (c2);
		\path (c1) edge node [above]{$u$} (c2);

		\node (c0) at (2,-0)   {{$a_{{n-1},{k-1}}$}}; 
		\node (c1) at (2.5,-1)   {{$a_{{n},{k}}$}}; 
		\node (c2) at (1.5,-1)   {{$a_{n,{k-1}} $}};
		\path (c0) edge node [left]{$-\frac{v}{u}$} (c2);
		\path (c1) edge node [above]{$\frac1u$} (c2);
		\end{tikzpicture}}
	\caption{Binomial interpolated triangle}
	\label{fig:construction_}
\end{figure}

\begin{theorem}\label{th:a_nk}
For all integer $r$, $k$ such that $k_0\leq k\leq n$, where $k_0\geq0$ is a fixed integer, we have
\begin{equation}\label{eq:interpolated_transf_gen}
a_{n,k}=\sum_{i=0}^{k-k_0}\binom{k-k_0}{i} u^{i}v^{k-k_0-i}a_{{n-k+k_0+i},{k_0}}.
\end{equation} 
\end{theorem}
\begin{proof}
We prove it by induction on $n$ and $k$. If $k={k_0}$, then the formula \eqref{eq:interpolated_transf_gen} trivially holds for any $n$. We suppose that \eqref{eq:interpolated_transf_gen} is true up to $k_0\leq k-1$. Let $1\leq\overline{k}=k-k_0$. Then

\begin{eqnarray*}
 a_{n,k}&=& u a_{{n},{k-1}}+v a_{{n-1},{k-1}} \\
 &=& u \sum_{i=0}^{\overline{k}-1}\binom{\overline{k}-1}{i} u^{i}v^{\overline{k}-i-1}a_{{n-\overline{k}+i+1},{k_0}} + v \sum_{i=0}^{\overline{k}-1}\binom{\overline{k}-1}{i}u^{i}v^{\overline{k}-i-1}a_{{n-\overline{k}+i},{k_0}} \\
 &=& \sum_{i=0}^{\overline{k}-1} \binom{\overline{k}-1}{i} u^{i+1}v^{\overline{k}-i-1}a_{{n-\overline{k}+i+1},{k_0}} + \sum_{i=0}^{\overline{k}-1} \binom{\overline{k}-1}{i} u^{i}v^{\overline{k}-i}a_{{n-\overline{k}+i},{k_0}} \\
 &=& \sum_{i=1}^{\overline{k}} \binom{\overline{k}-1}{i-1} u^{i}v^{\overline{k}-i}a_{{n-\overline{k}+i},{k_0}} + \sum_{i=0}^{\overline{k}-1} \binom{\overline{k}-1}{i} u^{i}v^{\overline{k}-i}a_{{n-\overline{k}+i},{k_0}} \\ 
 &=&u^{\overline{k}}a_{{n},{k_0}}+  \sum_{i=1}^{\overline{k}-1} \binom{\overline{k}-1}{i-1} u^{i}v^{\overline{k}-i}a_{{n-\overline{k}+i},{k_0}}  + 
 v^{\overline{k}}a_{{n-\overline{k}},{k_0}}+ \sum_{i=1}^{\overline{k}-1} \binom{\overline{k}-1}{i} u^{i}v^{\overline{k}-i}a_{{n-\overline{k}+i},{k_0}} \\
 &=&v^{\overline{k}}a_{{n-\overline{k}},{k_0}} + 
  \sum_{i=1}^{\overline{k}-1}\left( \binom{\overline{k}-1}{i-1}+ \binom{\overline{k}-1}{i} \right) u^{i}v^{\overline{k}-i}a_{{n-\overline{k}+i},{k_0}} 
+ u^{\overline{k}}a_{{n},{k_0}}\\  
 &=& \sum_{i=0}^{\overline{k}}\binom{\overline{k}}{i} u^{i}v^{\overline{k}-i}a_{{n-\overline{k}+i},{k_0}}.
\end{eqnarray*} 
\end{proof}

Considering the substitutions $k_0=0$ and $k=n$, or considering a fixed term $a_{{n_0},{k_0}}$ leads us to the following corollary.
\begin{corollary}
The right diagonal sequence is the binomial interpolated transform of the left diagonal sequence, so
\begin{equation}\label{eq:interpolated_transf}
b_n=a_{n,n}=\sum_{i=0}^{n}\binom{n}{i} u^iv^{n-i}a_{i,0}.
\end{equation}
Furthermore, let us fix $k_0$ and $n_0$, so that $0\leq k_0\leq n_0$. Then the terms $\overline{a}_{i,j}=a_{n,k}$ form a binomial interpolated sub-triangle (see Figure~\ref{fig:sub_triangle}), where  $i=n-n_0$, $j=k-k_0$ $n_0\leq n$, $k_0\leq k \leq n-(n_0-k_0)$. The sub-triangle's right diagonal sequence is the binomial interpolated transform of its left diagonal sequence.   
\end{corollary}

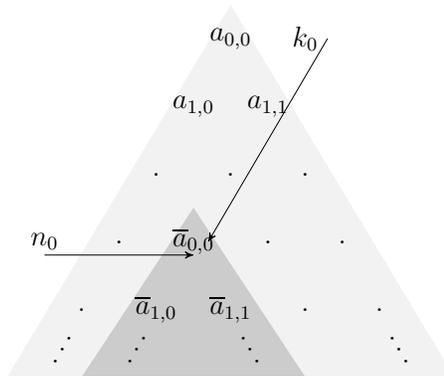
\begin{figure}
		\centering 
\newcommand*{\n}{1}
\newcommand*{\m}{2}
\renewcommand*{\N}{4}
\scalebox{0.90}{ 		         
	\begin{tikzpicture}[->,xscale=1.1,yscale=1.0, auto,swap]

	\node (r0) at ( 0,  0.5) {}; 
	\node (s0) at (-3, -5.0) {}; 
	\node (s1) at ( 3, -5.0) {}; 
	\node (si) at (-1.0, -4.0) {}; 
	\fill[fill=gray!10] (r0.center)--(s0.center)--(s1.center);
	\path[] (r0)--(s0);
	\path[] (s0)--(si);
	\path[] (si)--(s1);
	\path[] (s1)--(r0);	
	
	\foreach \i in {0,...,\n}
	\foreach \j in {0,...,\i} 
	{\node(a\i \j)  at ({-\i/2 + \j},{-\i}) {$a_{{\i},{\j}}$};}
	\foreach \i in {\m,...,\N}
	\foreach \j in {0,...,\i} 
	{\node(a\i \j)  at ({-\i/2 + \j},{-\i}) {$.$};}	

	\node (r0) at ( -0.5,  -2.5) {}; 
	\node (s0) at (-2.0, -5.0) {}; 
	\node (s1) at ( 1.0, -5.0) {}; 
	\node (si) at (-1.0, -4.0) {}; 
	\fill[fill=gray!40] (r0.center)--(s0.center)--(s1.center);
	\path[] (r0)--(s0);
	\path[] (s0)--(si);
	\path[] (si)--(s1);
	\path[] (s1)--(r0);		

    \node (r) at ( -2.5,  -3) {$n_0$};	
    \node (d) at ( 1,  0) {$k_0$}; 
    \path[draw]  ( -2.5, -3.2) -- (-0.5,-3.2) ;
    \path[draw]  ( 1.3, 0) -- (-0.3,-3) ;
	\foreach \i in {0,...,1}
	\foreach \j in {0,...,\i} 
	{\node(a\i \j)  at ({-\i/2 + \j-0.5},{-\i-3}) {$\overline{a}_{{\i},{\j}}$};}
	
	\pgfmathtruncatemacro{\i}{\N +1 };
	\node[rotate=25](a\i 0)  at ({-\i/2+0.2 },{0.5-\i}) {\reflectbox{$\ddots$}};
	\node[rotate=25](a\i 1)  at ({-\i/2+1.2 },{0.5-\i}) {\reflectbox{$\ddots$}};
	\pgfmathtruncatemacro{\j}{(\N + 1) }
	\node[rotate=-25] (a\i \j)  at ({-\i/2 +\j-0.2},{0.5-\i}) {$\ddots$};	
	\node[rotate=-25] (a\i \j)  at ({-\i/2 +\j-2.2},{0.5-\i}) {$\ddots$};

	\end{tikzpicture}}
\caption{Sub-triangle of the binomial interpolated triangle}
\label{fig:sub_triangle}
\end{figure}

We now express the terms $a_{n,k}$ by the right diagonal sequence $\{a_{n,n}\}_{n=0}^{\infty}$.
\begin{theorem}\label{th:a_nk_nn}
	For any $0\leq k\leq n$
	\begin{equation}\label{eq:inverse_interpolated_transf_gen}
	a_n^k=\sum_{i=0}^{n-k}\binom{n-k}{i}\left(\frac1u \right) ^{i}\left(-\frac{v}{u} \right) ^{n-k-i}a_{k+i}^{k+i}.
	\end{equation} 
\end{theorem}
\begin{proof} We suppose that the sequence  $\{b_n=a_{n,n}\}$ with $\{b_n\}_{n=0}^{\infty}=\{a_{n,n}\}_{n=0}^{\infty}$ is given. 
Let $U=1/u$ and $V=-v/u$, then  the right equality is 
$a_{{n},{k-1}}=Ua_{n,k}+Va_{{n-1},{k-1}}$, according to \eqref{eq:def_triangle_rule}. Based on this connection, we can write the entries of the binomial interpolated triangle (from right to left) by $b_{n,r}$, where $b_{n,0}=b_n$, $b_{n,r}=Ub_{n,{r-1}}+Vb_{{n-1},{r-1}}$ $(1\leq r\leq n)$.  Using relation \eqref{eq:interpolated_transf_gen}, proved in Theorem~\ref{th:a_nk}, we obtain  
$$b_{n,r} = \sum_{i=0}^{r}\binom{r}{i}U^{i}V^{r-i}b_{{n-r+i},0}.$$
Since $b_{n,r}=a_{n,{n-r}}$, considering the substitution $k=n-r$, the thesis follows.
\end{proof}

If $k=0$, as a direct consequence of Theorem~\ref{th:a_nk_nn},  the inverse transformation of the binomial interpolated transform \eqref{eq:interpolated_trans} is    

\begin{equation}\label{eq:inverse_interpolated_transf}
a_{n,0}=\sum_{i=0}^{n}\binom{n}{i}\left(\frac1u \right) ^{i}\left(-\frac{v}{u} \right) ^{n-i}a_{i,i}.
\end{equation}

\section{Binary binomial interpolated triangle}

Let $a_{n,0} = a_n$, where $\{a_n\}_{n=0}^{\infty}$ is a binary recursive sequence defined by
\begin{equation}\label{eq:def_binary_seq}
a_n=\alpha a_{n-1}+\beta a_{n-2}\qquad  ( 2\leq n,\  \alpha, \beta\in \mathbb{R},\  \alpha \beta\ne0), 
\end{equation} 
with  initial values $a_0, a_1\in \mathbb{R}$ ($|a_0|+|a_1|\ne 0 $). Then we call the binomial interpolated triangle  \textit{binary binomial interpolated triangle} and we let \BT$(a_0,a_1,\alpha,\beta;u,v)$ (in short \BT) denote it. 

Bhadouria et al.\ \cite{Bhadouria} gave three special examples, where the triangles are generated by the 4-Lucas sequence. They are \BT$(2,4,4,1;1,1)$, \BT$(2,4,4,1;4,1)$ and \BT$(2,4,4,1;4,4)$.

 From now on, we will use two important variables  
\begin{equation*}
	\mathcal{A}=u\alpha+2v\qquad\text{and}\qquad\mathcal{B}=u^2\beta -uv\alpha -v^2.
\end{equation*}

First of all, we give some recursive formulas for terms $a_{n,k}$. The results of the next technical lemma will be useful during the proofs of the further theorems.

\begin{lemma}
 The following recurrence relations hold 	 
\begin{alignat}{2}
a_{n,k} & = \alpha a_{{n-1},{k}}+\beta a_{{n-2},{k}} \qquad  &(2\leq k+2\leq n),\label{eq:t1}\\
\quad a_n^k & = \frac{u \beta}{v} a_{{n-1},{k}}-\frac{\mathcal{B}}{v}  a_{{n-1},{k-1}}\qquad &(2\leq k+1\leq n),\label{eq:t2}\\	
a_{n,k} & = \frac{u\alpha+v}{u} a_{{n-1},{k}}+\frac{\mathcal{B}}{u} a_{{n-2},{k-1}}\qquad &(2\leq k+1\leq n),\label{eq:t3}\\
\nonumber\\
a_{n,k} & = (u\alpha+v) a_{{n-1},{k-1}}+u\beta a_{{n-2},{k-1}}\qquad &(2\leq k+1\leq n),\label{eq:t4}\\
a_{n,k} & = \frac{u^2\beta +v^2}{v} a_{{n-1},{k-1}}- \frac{u\mathcal{B}}{v} a_{{n-1},{k-2}}\qquad &(2\leq k\leq n),\label{eq:t5}\\
a_{n,k} & = \mathcal{A}\, a_{{n-1},{k-1}}+ \mathcal{B}\, a_{{n-2},{k-2}}\qquad &(2\leq k\leq n),\label{eq:t6}\\    
\nonumber\\
a_{n,k} & = \frac{2u\beta-v\alpha}{\beta} a_{{n},{k-1}}-\frac{\mathcal{B}}{\beta}  a_{{n},{k-2}}\qquad &(2\leq k\leq n).\label{eq:t7} 
\end{alignat}	
\end{lemma}
\begin{proof}
	First we prove relation \eqref{eq:t1} by induction. For $k=0$ it corresponds to definition \eqref{eq:def_binary_seq}. Let us suppose that this formula is true up to $k-1$. So, if $x=a_{n-3,k-1}$ and $y=a_{n-2,k-1}$, then  $a_{n-1,k-1}= \beta x + \alpha y$ and $a_{n,k-1}=  \alpha\beta x+(\beta +{\alpha}^2)y$ hold. Figure~\ref{fig:steps01-f}, which is a suitable part of Figure~\ref{fig:construction_}, depicts it. 
	Now $a_{n-2,k}= v x + uy$, $a_{n-1,k}= u\beta  x+(u\alpha +v)y$ and $a_{n,k}= (u\alpha \beta +v \beta) x + (u {\alpha}^2 + v \alpha+ u \beta )y$, which gives \eqref{eq:t1}. Moreover, we proved with it, that Figure~\ref{fig:steps01-f} could be any part of  Figure~\ref{fig:construction_}. 
	
	In order to prove the further equations we can also use the relevant part of Figure~\ref{fig:steps01-f}. For example, in the case \eqref{eq:t4} $a_{n,k}=(u\alpha+v)\, y + u\beta\, x = u \beta  x +(u \alpha +v)y$.
\end{proof}

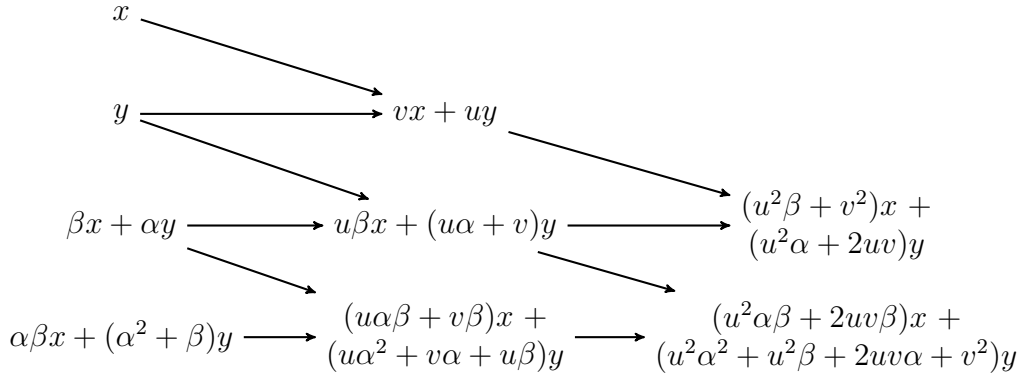
\begin{figure}[h!t] \centering
	\scalebox{0.99}{ \begin{tikzpicture}[->,xscale=3.0,yscale=1.5, auto,swap,every node/.style={shape=rectangle,draw=none}]

		\node(a0) at (0,-0.1)    {$x$};
		\node (a1) at (0,-1)   {$y$};
		\node (a2) at (0,-2)  {$\beta x + \alpha y $};
		\node (a3) at (0,-3)   {$ \alpha\beta x+({\alpha}^2 +\beta )y $};
		\node (b1) at (1.45,-1)   {$v x+u y$};
		\node (b2) at (1.45,-2)  {$u \beta  x +(u \alpha +v)y$};
		\node (b3) at (1.45,-3) [ align=center]   {$(u\alpha \beta +v \beta) x$ +  \\ $(u {\alpha}^2 + v \alpha+ u \beta )y$};
		\node (c1) at (3.2,-2)  [ align=center]  {$(u^2\beta+v^2)  x $ + \\ $(u^2\alpha +2uv)y$};
		\node (c2) at (3.2,-3)  [align=center] {$ (u^2\alpha\beta+2uv\beta) x$ + \\ $ (u^2{\alpha}^2+u^2\beta+2uv\alpha+v^2) y$};

		\path (a0) edge node[draw=none] {} (b1) [thick];
		\path (a1) edge node[draw=none] {} (b1) [thick];
		\path (a1) edge node[draw=none] {} (b2) [thick];
		\path (a2) edge node[draw=none] {} (b2) [thick];
		\path (a2) edge node[draw=none] {} (b3) [thick];
		\path (a3) edge node[draw=none] {} (b3)[thick];
		\path (b1) edge node[draw=none] {} (c1)[thick];
		\path (b2) edge node[draw=none] {} (c1)[thick];
		\path (b2) edge node[draw=none] {} (c2)[thick];
		\path (b3) edge node[draw=none] {} (c2)[thick];		
		\end{tikzpicture}}
	\caption{A part of the growing}
	\label{fig:steps01-f}
\end{figure}

The relation \eqref{eq:t6} gives that the right diagonal sequence $b$ (and the $i$-th diagonals parallel to it) satisfies the binary recurrence relation
\begin{equation}
 \forall n \geq 2\quad	b_n=\mathcal{A}b_{n-1}+\mathcal{B}b_{n-2}
\end{equation}
with initial values $b_0=a_0$, $b_1=ua_1+va_0$ (generally, $b_0=a_{i,0}$, $b_1=a_{i+1,1}$, $i\geq0$). Moreover, the system of equations $\mathcal{A}=0$, $\mathcal{B}=0$ gives  $v=-u\alpha/2$ and $u^2(\alpha^2+4\beta)= 0$. Thus, as the condition $uv\ne0$ holds, if  $v\ne-u\alpha/2$ or $\alpha^2+4\beta\ne0$, then $|\mathcal{A}|+|\mathcal{B}|\ne0$, otherwise $\mathcal{A}=\mathcal{B}=0$.

\subsection{Sums and alternating sums of rows}

Let  $s=\{s_n\}_{n=0}^{\infty}$ be the sequence whose elements are the sums of the values belonging to the $n$-th row of a binomial interpolated triangle. We have 
\begin{equation}\label{eq:sumrow}
s_n=\sum_{k=0}^{n}a_{n,k}=\sum_{k=0}^{n} \sum_{i=0}^{k}\binom{k}{i} u^{i}v^{k-i}a_{n-k+i,0}.
\end{equation}

We obtain a linear recurrence for $s$, whose coefficients depend only on $\alpha$, $\beta$, $\mathcal{A}$ and $\mathcal{B}$, i.e., the coefficients of the binary recurrences related to sequences $a$ and $b$.

\begin{theorem}
The sequence $s$ satisfies the	following fourth order linear homogeneous recurrence
\begin{equation}\label{eq:sn}
\forall n\geq 4\quad s_{n} = (\alpha+\mathcal{A})s_{n-1}+(\beta-\alpha \mathcal{A}+\mathcal{B})s_{n-2}-(\alpha \mathcal{B}+\beta \mathcal{A})s_{n-3}-\beta \mathcal{B} s_{n-4}.
\end{equation}	
\end{theorem}
\begin{proof} 
	The sum of row $n$ can be given by the previous two ones in form 
	\begin{equation}\label{eq:sumts0}
	s_n=\alpha (s_{n-1}-a_{n-1,n-1})+\beta s_{n-2}+a_{n,n-1}+a_{n,n}.
	\end{equation}	
	Since $u a_{n,n-1}=a_{n,n}- va_{n-1,n-1}$ then from \eqref{eq:sumts0} we obtain
	\begin{equation}\label{eq:sumt0}
	u\, s_n= u\alpha\, s_{n-1}+ u\beta\, s_{n-2}+(u+1) a_{n,n} -
	(\alpha u +v ) a_{n-1,n-1}.
	\end{equation}
	Applying  \eqref{eq:t6} and as $\mathcal{A}-\alpha u-v=v$ we have 
	\begin{equation}\label{eq:sumt1}
	u\, s_n= u\alpha \, s_{n-1}+ u\beta\, s_{n-2} + (\mathcal{A}u+v)  a_{n-1,n-1} + \mathcal{B}(u+1)  a_{n-2,n-2}. 
	\end{equation}
	Now we transform \eqref{eq:sumt0} into 
	\begin{equation}\label{eq:sumt2}
	u \,s_{n-1}= u\alpha\, s_{n-2}+ u\beta\, s_{n-3}+(u+1) a_{n-1,n-1} -
	(\alpha u +v ) a_{n-2,n-2}.
	\end{equation}
	Transforming again \eqref{eq:sumt0} into form $u\,s_{n-2}=u\alpha\, s_{n-3}+\ldots $ and using \eqref{eq:t6} we gain 
	\begin{equation}\label{eq:sumt3}
	u\mathcal{B}\, s_{n-2}= u\mathcal{B}(\alpha \, s_{n-3}+ \beta \, s_{n-4})
	- (\alpha u +v ) a_{n-1,n-1} 
	+\left( \mathcal{B}(u+1) +  (\alpha u +v )\mathcal{A} \right)  a_{n-2,n-2}.
	\end{equation}	
	
	Finally, let us express $a_{n-1,n-1}$ and $a_{n-2,n-2}$ from \eqref{eq:sumt1} and \eqref{eq:sumt2}, respectively. We substitute them into \eqref{eq:sumt3} and  obtain \eqref{eq:sn}.  
\end{proof}

\begin{remark}\label{rem:sn}
It should be noticed that the recurrence \eqref{eq:sumts0} is the minimal order recurrence of $s$ when 
$$\mathcal{B}(u+1)^2+(\alpha u+v)(\mathcal{A}u+v)\ne0.$$
Indeed, in this case the determinant of coefficients of the system (from \eqref{eq:sumt1} and \eqref{eq:sumt2})
$$
\begin{cases}
(\mathcal{A}u+v)  a_{n-1,n-1} + \mathcal{B}(u+1)  a_{n-2,n-2} =u\, s_n- u\alpha \, s_{n-1}- u\beta\, s_{n-2} \\
(u+1) a_{n-1,n-1} -
(\alpha u +v ) a_{n-2,n-2} =	u \,s_{n-1}- u\alpha\, s_{n-2}- u\beta\, s_{n-3}
\end{cases}
$$
is different from $0$, so the system has a unique solution, otherwise when
$$\mathcal{B}(u+1)^2+(\alpha u+v)(\mathcal{A}u+v)=0$$
or, equivalently, with a little bit of calculations, when $\mathcal{B}(u+1)^2+(\alpha-u)(v+u\alpha)=0$ the minimal order of the linear recurrence of $s$ should be less then four. For example, the solution $u=-1$ and $v=\alpha$ implies that $\mathcal{A}=\alpha$ and $\mathcal{B}=\beta$. Then the relation \eqref{eq:sumt1} becomes 
\begin{equation}\label{eq:sumss}
\forall n\geq 2\quad s_n= \alpha \, s_{n-1}+ \beta\, s_{n-2}. 
\end{equation}
Obviously, summing relations $s_n= \alpha  s_{n-1}+ \beta s_{n-2}$, $-\alpha s_{n-1}=- \alpha^2  s_{n-2}-\alpha \beta s_{n-3}$ and $-\beta s_{n-2}= -\alpha\beta  s_{n-3}- \beta^2 s_{n-4}$ we find that the recurrence \eqref{eq:sn} also holds for sequence $s$, but it is not the minimal order one. 
\end{remark}

Let the sequence $\bar{s}=\{\bar{s}_n\}_{n=0}^{\infty}$ be the alternating sum sequence, where the terms are the values of rows of a binomial interpolated triangle, so that
\begin{equation*}\label{eq:altsumrow}
\bar{s}_n=\sum_{k=0}^{n}(-1)^k a_{n,k}=\sum_{k=0}^{n} (-1)^k \sum_{i=0}^{k}\binom{k}{i} u^{i}v^{k-i}a_{{n-k+i},0}.
\end{equation*}

\begin{theorem}
The sequence $\bar{s}$ satisfies the following fourth order linear homogeneous recurrence 
\begin{equation}\label{eq:altsn}
\forall n\geq 4\quad{\bar{s}}_{n} = (\alpha-\mathcal{A})\bar{s}_{n-1}+(\beta+\alpha \mathcal{A}+\mathcal{B})\bar{s}_{n-2}-(\alpha \mathcal{B}-\beta \mathcal{A})\bar{s}_{n-3}-\beta \mathcal{B} \bar{s}_{n-4}.
\end{equation}
\end{theorem}
\begin{proof}
Row by row the signs of the entries in the alternating sums do not change in directions parallel to the left diagonal ($\sign((-1)^k a_{n,k})=\sign((-1)^k a_{n-1,k})$), but parallel to the right diagonal they change ($\sign((-1)^k a_{n,k})\ne\sign((-1)^{k-1} a_{n-1,k-1})$). Hence we only have to change the sign of $\mathcal{A}$ in the summation relation \eqref{eq:sn}.
\end{proof}

\begin{remark}
If  $u=-1$ and $v=\alpha$, then the relation \eqref{eq:altsn} becomes more simple, 
\begin{equation}\label{eq:sum_alterss}
 \forall n\geq 3\quad{\bar{s}}_{n} = (\alpha^2+2\beta)\bar{s}_{n-2}-\beta^2 \bar{s}_{n-4}.
\end{equation}
\end{remark}

\subsection{Rising diagonal sum sequence}

The sequence $d=\{d_n\}_{n=0}^{\infty}$ of sums of elements in rising diagonals has the following definition  
\begin{equation}
\forall n \geq 0\quad d_n=\sum_{k=\lceil\frac{n}{2}\rceil}^{n}a_{k,n-k}=\sum_{k=0}^{\lfloor\frac{n}{2}\rfloor}a_{n-k,k}.
\end{equation}

When we consider Pascal's arithmetical triangle it is well-known that these sums provide the Fibonacci sequence. We give the recurrence relation for the sequence of sums of elements belonging to rising diagonal (i.e.,``shallow diagonal" as they are defined in Wolfram Math World \cite{wolfram}) of our binomial interpolated triangles.  

\begin{theorem} The rising diagonals sums sequence $d$ of a \BT satisfies the sixth order linear recurrence relation 
	\begin{equation}\label{eq:Dn}
\forall n\geq6 \quad 	D_n=\mathcal{A}D_{n-2}+\mathcal{B}D_{n-4},
	\end{equation}
	where  $D_n=-d_n+\alpha d_{n-1}+\beta d_{n-2}$.
	Moreover, for even $n$ ($n=2k$, $n\geq2$), $D_n=-b_k$  also holds. 
\end{theorem}

We point out that with the equation \eqref{eq:Dn} is a concise expression for the following sixth order recurrence relation   
\begin{equation}d_{n} = \alpha d_{n-1}+(\beta+\mathcal{A}) d_{n-2}-\alpha \mathcal{A}  d_{n-3}+
(-\beta \mathcal{A}+\mathcal{B}) d_{n-4} - \alpha \mathcal{B}  d_{n-5} - \beta \mathcal{B} d_{n-6}.
\end{equation}

\begin{proof}
Without loss of generality, let $n=2k$. 
	The values of $d_{n+2}$ and $d_{n+3}$ can be given by the previous two ones in forms 
	\begin{align}
	d_{n+2} \ &=\  \alpha d_{n+1}+\beta d_{n}+ a_{k+1,k+1},\nonumber\\
	d_{n+3} \ &=\  \alpha \left(d_{n+2}-a_{k+1,k+1}\right)+ \beta d_{n+1}+a_{k+2,k+1},\nonumber\\
	\intertext{generally for even and odd indexes we have}
	d_{n+2i} \ &=\  \alpha d_{n+2i-1}+\beta d_{n+2i-2}+ a_{k+i,k+i},\nonumber\\
	d_{n+2i+1} \ &=\  \alpha \left(d_{n+2i}-a_{k+i,k+i}\right)+ \beta d_{n+2i-1}+a_{k+i+1,k+i}.\nonumber
	\end{align}

When we consider $D_n$, with $n$ even, we obtain
\begin{equation*}
\forall i\geq1\quad	a_{{k+i},{k+i}} = d_{n+2i}- \alpha d_{n+2i-1}-\beta d_{n+2i-2}=-D_{n+2i},
\end{equation*} 
	and according to  \eqref{eq:t6}
\begin{equation*}
	D_{n+6}-\mathcal{A}D_{n+4}-\mathcal{B}D_{n+2}= -a_{k+3,k+3}+ \mathcal{A} a_{k+2,k+2}+ \mathcal{B} a_{k+1,k+1}=0.
\end{equation*}
	
If we deal with the case $D_n$, with $n$ odd, then, using relations \eqref{eq:t3} and \eqref{eq:def_triangle_rule},  we find
	\begin{eqnarray*}
		d_{n+3}&=&\alpha d_{n+2}+ \beta d_{n+1}+ \frac{v}{u} a_{k+1,k+1} + \frac{\mathcal{B}}{u} a_{k,k},\label{eq:sumd_n3}\\
		d_{n+5}&=&\alpha d_{n+4}+ \beta d_{n+3}+ \frac{v\mathcal{A}+\mathcal{B}}{u} a_{k+1,k+1} + \frac{v\mathcal{B}}{u} a_{k,k},\label{eq:sumd_n5}\\
		d_{n+7}&=&\alpha d_{n+6}+ \beta d_{n+5}+ \frac{v\mathcal{A}^2+v\mathcal{B}+\mathcal{A}\mathcal{B}}{u} a_{k+1,k+1} + \frac{v\mathcal{A}\mathcal{B}+\mathcal{B}^2}{u} a_{k,k}.\label{eq:sumd_n7}
	\end{eqnarray*}
Thus 
	\begin{multline*}
	u(D_{n+7}-\mathcal{A}D_{n+5}-\mathcal{B}D_{n+3})= \\
	(-(v\mathcal{A}^2+v\mathcal{B}+\mathcal{A}\mathcal{B})+\mathcal{A}(v\mathcal{A}+\mathcal{B})+v\mathcal{B}) a_{k+1,k+1} \\ +(-(v\mathcal{A}\mathcal{B}+\mathcal{B}^2)+v\mathcal{A}\mathcal{B}+\mathcal{B}^2 ) a_{k,k}=0.
	\end{multline*}
\end{proof}

\subsection{Central elements}

The central elements of the binomial triangles are the terms $a_{2k,k}$ with $k\geq0$. In this subsection we give the relation between them, moreover it turns out that the before mentioned recurrence relation holds for all the columns (which are parallel to the vertical axis of the triangle).
These sequences are defined by $\{a_{{2k+\ell},k}\}_{k=k_0}^{\infty}$ with $\ell\in \mathbb{Z}$ and
\begin{equation}\label{eq:2k+l}
k_0=\begin{cases}
\ 0,\quad &\text{ if } \ell\geq0;\\
\lvert\ell\lvert, \quad &\text{ if } \ell<0.
\end{cases}
\end{equation}

Moreover,  if $\ell=0$, then  it is the sequence of the central elements. 

\begin{theorem} All the sequences $\{a_{{2k+\ell},k}\}_{k=k_0}^{\infty}$ of the binomial interpolated triangle with $\ell\in \mathbb{Z}$ and $k_0$ defined by \eqref{eq:2k+l}
 satisfy the same binary homogeneous recurrence relation 
\begin{equation}\label{eq:cent}
c_{k+2}=(\alpha^2u+\alpha v+2\beta u)c_{k+1} -\beta \mathcal{B}c_{k},
\end{equation}
where $c_k=a_{{2k+\ell},k}$.
\end{theorem}
\begin{proof}
We have to prove that
\begin{equation}\label{eq:diag_a}
a_{{2(k+2)+\ell},{k+2}}=(\alpha^2u+\alpha v+2\beta u)a_{{2(k+1)+\ell},{k+1}} -\beta(u^2\beta -uv\alpha -v^2) a_{2k+\ell,k}. 
\end{equation}	

First, supposing that $u\alpha+v\ne 0$ and recalling \eqref{eq:t4}  and \eqref{eq:t1} we obtain
\begin{eqnarray*}
a_{{2k+3+\ell},{k+1}}&=& (u\alpha+v)a_{{2k+2+\ell},{k}} + u\beta a_{{2k+1+\ell},{k}} \\
 &=& (u\alpha+v)\left(\alpha a_{{2k+1+\ell},{k}} + \beta a_{{2k+\ell},{k}} \right) + u\beta a_{{2k+1+\ell},{k}}\\
 &=& (u\alpha^2+v\alpha+u\beta)\frac{a_{{2k+2+\ell},{k+1}}-u\beta a_{{2k+\ell},{k}}}{u\alpha+v} + (u\alpha+v)\beta a_{{2k+\ell},{k}},
\end{eqnarray*}
furthermore
\begin{eqnarray*}
a_{{2k+4+\ell},{k+2}}&=& (u\alpha+v)a_{{2k+3+\ell},{k+1}} + u\beta a_{{2k+2+\ell},{k+1}} \\
&=& (u\alpha^2+v\alpha+u\beta)\left(a_{{2k+2+\ell},{k+1}}-u\beta a_{{2k+\ell},{k}} \right) + (u\alpha+v)^2 \beta a_{{2k+\ell},{k}}\\
& &\qquad + u\beta a_{{2k+2+\ell},{k+1}} \\
&=& (u\alpha^2+v\alpha+2u\beta)a_{{2k+2+\ell},{k+1}}+ \beta (-u^2\beta +uv\alpha +v^2) a_{{2k+\ell},{k}}. 
\end{eqnarray*}

Second, when  $v=-u\alpha$ the equality \eqref{eq:cent} becomes 
\begin{equation*}
c_{k+2}=2u\beta c_{k+1} -u^2\beta^2 c_{k},
\end{equation*}
or explicitly 
\begin{equation}\label{eq:cen3}
a_{{2(k+2)+\ell},{k+2}}=2u\beta a_{{2(k+1)+\ell},{k+1}} -u^2\beta^2 ca_{{2k+\ell},{k}}.
\end{equation}
Thus, since the sequence is a geometric progression we have
\begin{equation*}
a_{{2(k+2)+\ell},{k+2}}
=u\beta a_{{2(k+1)+\ell},{k+1}}
=u^2\beta^2 a_{{2k+\ell},{k}} ,
\end{equation*}
and clearly \eqref{eq:cen3} holds because it corresponds to the identity
\begin{equation*}
u\beta a_{{2(k+1)+\ell},{k+1}}
=2u\beta a_{{2(k+1)+\ell},{k+1}}
-u\beta a_{{2(k+1)+\ell},{k+1}}.
\end{equation*}
 \end{proof}

\subsection{Explicit formula}

\begin{theorem}\label{th:explicit}
 Let $D=\sqrt{\alpha^2+4\beta}$.  If $D\ne0$, then the explicit formula of $a_{n,k}$   is
	\begin{multline}
	a_{n,k}= \frac{(vD+\alpha v-2\beta u)a_{n,0}+2\beta a_{n,1}}{2 vD}\left(\frac{\beta u-vx_2}{\beta}\right)^k + \\ \frac{(vD-\alpha v +2\beta u)a_{n,0}-2\beta a_{n,1}}{2 vD}\left(\frac{\beta u-vx_1}{\beta}\right)^k,\label{eq:explicit}
	\end{multline}
	where  $x_1=(\alpha+D)/2$, $x_2=(\alpha-D)/2$  and
	\begin{eqnarray}
	a_{n,0} &=& \frac{(D-\alpha)a_{0,0}+2a_{1,0}}{2D}x_1^n+ \frac{(D+\alpha)a_{0,0}-2a_{1,0}}{2D}x_2^n,\label{eq:exp_an0}\\
	a_{n,1} &=& \frac{(D-\alpha)a_{0,0}+2a_{1,0}}{2D}x_1^{n-1}(ux_1+v)+ \frac{(D+\alpha)a_{0,0}-2a_{1,0}}{2D}x_2^{n-1}(ux_2+v).\label{eq:exp_an1}
	\end{eqnarray}
	
If $D=0$ and $\mathcal{A}\ne0$, then  the explicit formula is 
	\begin{equation*}
	a_{n,k}= \left( a_{n,0}+k\left(\frac{\alpha a_{n,1}}{\mathcal{A}} -a_{n,0} \right) \right) 
	{\left( \frac{\mathcal{A}}{\alpha}\right) }^k,
	\end{equation*}
	where 
	\begin{eqnarray}
	a_{n,0} &=& \left(a_{0,0}+n\frac{2a_{1,0}-\alpha a_{0,0}}{\alpha } \right) \left( \frac{\alpha}{2}\right)^n \label{eq:exp_an0_2}\\
	a_{n,1} &=& ua_{n,0}+v \left(a_{0,0}+(n-1)\frac{2a_{1,0}-\alpha a_{0,0}}{\alpha } \right) \left( \frac{\alpha}{2}\right)^{n-1}.\label{eq:exp_an1_2}
	\end{eqnarray}

If $D=0$ and $\mathcal{A}=0$, then 
 \begin{eqnarray}
  a_{n,0} &=& \left(a_{0,0}+n\frac{2a_{1,0}-\alpha a_{0,0}}{\alpha } \right) \left( \frac{\alpha}{2}\right)^n \label{eq:exp_an0_3}\\
  a_{n,1} &=& u\frac{2a_{1,0}-\alpha a_{0,0}}{\alpha }  \left( \frac{\alpha}{2}\right)^{n}.\label{eq:exp_an1_3}
\end{eqnarray}
and  $a_{n,k}=0$, if $k\geq2$. 
\end{theorem}

\begin{proof}
We suppose that $\alpha^2+4\beta\ne0$. Firstly, we give the explicit form of the elements in case of the main path ($k=0$). As the recursion is binary with coefficient $\alpha$ and $\beta$, then the characteristic equation of \eqref{eq:def_binary_seq} is $x^2-\alpha x-\beta=0$  with roots $x_1\ne x_2$. Consequently, for all $n\geq 0$ the term $a_{n,0}$ is a linear combination of the powers $x_1^n$, $x_2^n$. In order to determine the coefficients of this linear combination, we need to solve the system of equations $px_1^i+qx_2^i=a_{i,0}$, $i=0,1$. From this system we find $p=\frac{a_{1,0}-a_{0,0}x_2}{x_1-x_2}$ and $q=\frac{a_{1,0}x_1-a_{1,0}}{x_1-x_2}$. Thus observing that 
$$x_1=\frac{a+D}{2}, x_2=\frac{a-D}{2}, x_1-x_2=D,$$
where $D=\sqrt{\alpha^2+4\beta}$, we easily find equality \eqref{eq:exp_an0}. 
Moreover $a_{n,1}=ua_{n,0}+va_{n-1,0}$ yields \eqref{eq:exp_an1}.
 
Secondly, we give similarly the explicit formula of the sequence $\{a_{n,k}\}_{k=0}^n$  with initial conditions $a_{n,0}$ and $a_{n,1}$. Because of \eqref{eq:t7}, its characteristic equation is 
$$y^2-\frac{2\beta u-\alpha v}{\beta}y -\frac{\alpha uv-\beta u^2+v^2}{\beta}=0$$
and the roots  are
\begin{equation*}
y_1=\frac{2\beta u-\alpha v+vD}{2\beta}=u + \frac{v(D-\alpha)}{2\beta}= \frac{\beta u-v x_2}{\beta}\quad \text{and}\quad y_2=\frac{\beta u-v x_1}{\beta}.
\end{equation*}

When we consider the case $D^2=\alpha^2+4\beta=0$ and $\mathcal{A}\ne0$ (or, equivalently,  $\beta=-\alpha^2/4$ and $v=-u\alpha/2$) we use the same method adopted before. 
Taking in account that the characteristic equation $x^2-\alpha x-\beta=0$ has the unique root  $x_0=\alpha/2$, we have to solve the system $(p+qi)x_0^i=a_{i,0}$, $i=0,1$, which provides the coefficients $p$, $q$ for $a_{n,0}=(p+qn)x_0^n$ in equality \eqref{eq:exp_an0_2}.
In this case the characteristic equation of sequence $\{a_n^k\}_{k=0}^{n}$ also has just one root $y_0=\mathcal{A}/\alpha\ne0$   and the equations $(\bar{p}+\bar{q}i)y_0^i=a_n^i$, $i=0,1 \text{ and } k$ yield the final formula.  

Finally, we examine the case $D^2=\alpha^2+4\beta=0$ and $\mathcal{A}=0$. Now the equivalent conditions $\beta=-\alpha^2/4$ and $v=-u\alpha/2$ imply that  $\mathcal{B}=0$ and the equality \eqref{eq:exp_an1_2} is simplified into \eqref{eq:exp_an1_3}.
Moreover if $k\geq2$, then the relation \eqref{eq:t7} becomes $a_{n,k}=0\,a_{n,k-1}+0\,a_{n,k-2}=0$.
\end{proof}

\section{Special types of binary binomial interpolated triangles}

The classical Pascal's triangle has vertical symmetry and its inner elements satisfy the well-known rule of addition, namely every elements is the sum of the two terms directly above it. In this section we give the classes of our triangles which have the same properties, and we answer the ``Whether Pascal's triangle is a binomial interpolated triangle?" question.

A binomial interpolated triangle is (vertically) symmetrical if  
\begin{equation*}\label{eq:def_triangle_symmetrical}
 a_{n,k}=a_{n,n-k} \quad  (0\leq k\leq n).
\end{equation*}

\begin{theorem}\label{th:sym}
Let $a_{0,0}\ne0$ and $a_{1,0}$ be given.
A binary binomial interpolated triangle is symmetrical if and only if 
$$u=-1,\  v=\alpha=\frac{2a_{1,0}}{a_{0,0}}$$
or 
$$\alpha=\frac{2a_{1,0}}{a_{0,0}},\  \beta=-\frac{\alpha^2}{4},\  v=-\frac{\alpha(u-1)}{2},$$
where $u\ne1$ or 
$$u=-1,\  v=\alpha=\frac{2a_{1,0}}{a_{0,0}},\  \beta=-\frac{\alpha^2}{4}.$$
(See \cite{Barbero} for the binomial transform in a special case). 
\end{theorem}
\begin{proof}
Indeed clearly the condition $a_{n,k}=a_{n,n-k}$ implies the following system of equations
$$\begin{cases}
\alpha = \mathcal{A}=u\alpha+2v\\
\beta = \mathcal{B}=u^2\beta-uv\alpha-v^2\\
b_1=a_{1,0}=ua_{1,0}+va_{0,0}
\end{cases}.$$

From the first equation we find $v=-\alpha(u-1)/2$  and substituting in the second equation we obtain, with some calculations,
$$(\alpha^2+4\beta)(u+1)(u-1)=0.$$

Now, we have the three possibilities $\alpha^2+4\beta=0$, $u=-1$, $u=1$.
Obviously the case $u=1$ implies $v=0$, a contradiction. The case $u=-1$ gives $v=\alpha$ and  $\alpha=2a_{1,0}/a_{0,0}$.
Considering the case $\alpha^2+4\beta=0$ we obtain the second part of the statement of the theorem. Finally, the equalities $\alpha^2+4\beta=0$ and $u=-1$ yield the third last case. 
\end{proof}

\begin{corollary}
Let $\lambda =a_{1,0}/a_{0,0}$ and  $a_{0,0}\ne0$, $u\ne 1$. Then the form of a symmetrical binomial interpolated triangle can be written by 
\begin{equation}\label{tri:sym}
\text{\BT}(a_0,\lambda a_0,2\lambda,\beta;-1,2\lambda),
\end{equation}
\begin{equation}\label{tri:sym2}
\text{\BT}(a_0,\lambda a_0,2\lambda,-\lambda^2;u,\lambda(1-u)),
\end{equation}
or
\begin{equation}\label{tri:sym3}
\text{\BT}(a_0,\lambda a_0,2\lambda,-\lambda^2;-1,2\lambda).
\end{equation}
\end{corollary}
Figure~\ref{fig:Triangle_from_symmetrical} shows a symmetrical binomial interpolated triangle, namely \BT$(2,1,1,1;-1,1)$ generated by Lucas numbers ({A000032} in \cite{Sloane}).

\begin{figure}[!ht]
	\centering
		\scalebox{0.99}{ \begin{tikzpicture}[->,xscale=0.8,yscale=0.7, auto,swap]
			\node(a00) at (0,0)    {2};
			
			\node (a01) at (-0.5,-1)   {1};
			\node (a02) at (0.5,-1)   {1};
			
			\node (a01) at (-1,-2)   {3};
			\node (a02) at (0,-2)   {$-2$};	
			\node (a01) at (1,-2)   {3};
			
			\node (a02) at (-1.5,-3)   {4};		
			\node (a01) at (-0.5,-3)   {$-1$};
			\node (a01) at (0.5,-3)   {$-1$};
			\node (a02) at (1.5,-3)   {4};	
			
			\node (a01) at (-2,-4)   {7};
			\node (a01) at (-1,-4)   {$-3$};
			\node (a01) at (0,-4)   {2};
			\node (a02) at (1,-4)   {$-3$};
			\node (a02) at (2,-4)   {7};
			
			\node (a01) at (-2.5,-5)   {11};		
			\node (a01) at (-1.5,-5)   {$-4$};
			\node (a02) at (-0.5,-5)   {1};
			\node (a02) at (0.5,-5)   {1};
			\node (a01) at (1.5,-5)   {$-4$};
			\node (a01) at (2.5,-5)   {11};
			
			\node (a01) at (-3,-6)   {18};
			\node (a02) at (-2,-6)   {$-7$};
			\node (a02) at (-1,-6)   {3};
			\node (a01) at (0,-6)   {$-2$};
			\node (a01) at (1,-6)   {3};
			\node (a02) at (2,-6)   {$-7$};
			\node (a01) at (3,-6)   {18};
			
			\end{tikzpicture}}
		\caption{Symmetrical binomial interpolated triangle \BT$(2,1,1,1;-1,1)$}
		\label{fig:Triangle_from_symmetrical}
\end{figure}

\begin{remark}
According to \eqref{eq:sumss} the recursion of $s$ in the symmetrical binomial interpolated triangles case is $s_n= 2\lambda s_{n-1}+ \beta s_{n-2}$ or $s_n= 2\lambda s_{n-1}-\lambda^2 s_{n-2}$. Moreover, when we consider the triangle \eqref{tri:sym3} for the sequence $d$ we obtain $d_n=\lambda d_{n-1}+\lambda d_{n-2}-\lambda^2 d_{n-3}$.
\end{remark}

\begin{theorem}
Pascal's triangle is not a binomial interpolated triangle.
\end{theorem}
\begin{proof}
The ``sum of above terms" condition and relation \eqref{eq:t2} should imply $a\beta/v=-\mathcal{B}/v=1$ and the vertical symmetry implies $\beta=\mathcal{B}$. Thus we find $u=-1$ and $v=-\beta$. Now the first set of solutions coming from Theorem~\ref{th:sym} gives $v=\beta=\alpha=2$, since in Pascal's triangle we have $a_{0,0}=a_{1,0}=1$. The recurrence \eqref{eq:def_binary_seq} becomes $a_{n,0}=2a_{n-1,0}-2a_{n-2,0}$ which is clearly not satisfied in Pascal's triangle by the first elements (all equal to 1) of any row. From the second set of solutions we find $\alpha=2$, $\beta=-1$, $v=1-u$, but, since $u=-1$, from these solutions we must have $v=2$ and on the other hand since $v=-\beta=-1$ we find a contradiction.
\end{proof}

Now we give the conditions of $u$ and $v$, so that an inner entry of \BT\ could be the  sum of values not only left but also directly above it by coefficients $u$ and $v$. Using  \eqref{eq:t2}, we can gain it in two different ways (see Figure~\ref{fig:sum}). Our cases are 
\begin{alignat*}{2}
\text{\textbf{case 1:}}\qquad \qquad & u=\frac{u\beta}{v}, \qquad  & v =-\frac{\mathcal{B}}{v},\\
\text{\textbf{case 2:}}\qquad \qquad   & v=\frac{u\beta}{v}, &u =-\frac{\mathcal{B}}{v}.	
\end{alignat*}

\begin{figure}[!ht]
	\centering 
	\scalebox{0.90}{ 		         
		\begin{tikzpicture}[->,xscale=1.7,yscale=1.3, auto,swap]

		\node (c0) at (-2,0)   {$a_{n-1,k-1} $}; 
		\node (c1) at (-1,0)   {$a_{n-1,k} $}; 
		\node (c2) at (-1.5,-1)   {{$a_{n,k} $}};
		\path (c0) edge node [right]{$v$} (c2);
		\path (c1) edge node [right]{$u$} (c2);

		\node (c0) at (1,0)   {$a_{n-1,k-1} $}; 
		\node (c1) at (2,0)   {$a_{n-1,k} $}; 
		\node (c2) at (1.5,-1)   {{$a_{n,k} $}};
		\path (c0) edge node [right]{$u$} (c2);
		\path (c1) edge node [right]{$v$} (c2);
		\end{tikzpicture}}
	\caption{Coefficients of summing downwards}
	\label{fig:sum}
\end{figure}
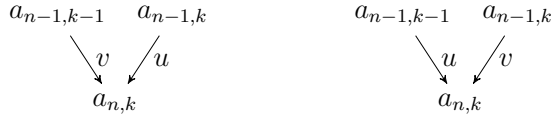

\subsection{Case 1}

Solving the system 
$$
\begin{cases}
u=\frac{u\beta}{v}\\
v =-\frac{\mathcal{B}}{v}
\end{cases},
$$
we can easily obtain that $u=\alpha$ and $v=\beta$, thus the triangle is 
\begin{equation}\label{tri:case1}
	\text{\BT}(a_0,a_1,\alpha,\beta;\alpha,\beta).
\end{equation}

We derive some interesting properties of this triangle. From Figure~\ref{fig:Triangle_alpha_beta}, which shows the first four rows of triangle \BT$(a_0,a_1,\alpha,\beta;\alpha,\beta)$ (see also Figure~\ref{fig:same}), we notice that the rows and the left diagonal satisfy the same recurrence. 
Indeed, using the equalities $u=\alpha$, $v=\beta$, the recurrence relation \eqref{eq:t7} becomes  $a_{n,k} = \alpha a_{{n},{k-1}}+{\beta}a_{{n},{k-2}}$. 
 Moreover, the terms along the rising diagonals are the same. The next theorem and corollary will provide it precisely.

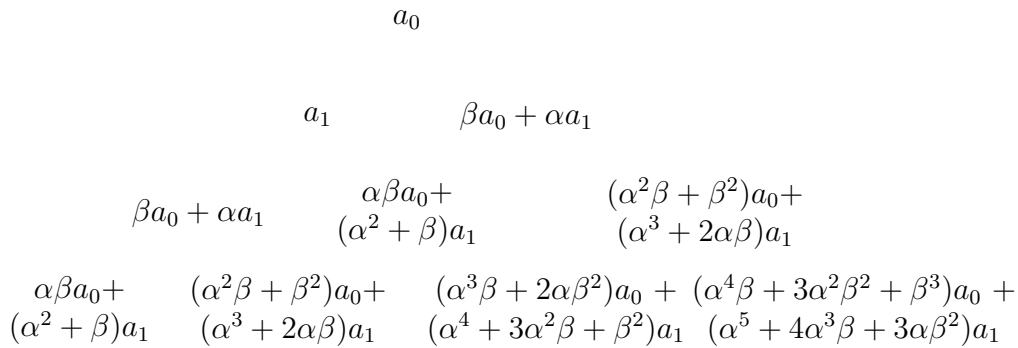
\begin{figure}
	\centering 
	\scalebox{0.99}
	{ \begin{tikzpicture}[->,xscale=4,yscale=1.3, auto,swap]
		\node(a00) at (0.,0)    {$a_0$};
		
		\node (a01) at (-0.3,-1)   {$a_1$};
		\node (a02) at (0.4,-1)   {$\beta a_0+\alpha a_1$};
		
		\node (a01) at (-0.7,-2)   {$\beta a_0+\alpha a_1$};
		\node (a02) at (0.,-2)  [ align=center]  {$\alpha \beta a_0 $+\\ $({\alpha}^2+\beta)a_1$};	
		\node (a01) at (1,-2) [ align=center]  {$({\alpha}^2 \beta+{\beta}^2) a_0 $+\\ $({\alpha}^3+2\alpha\beta)a_1$};
		
		\node (a02) at (-1.1,-3)  [ align=center] {$\alpha \beta a_0 $+\\ $ ({\alpha}^2+\beta)a_1$};	
		\node (a01) at (-0.4,-3)  [ align=center]  {$({\alpha}^2 \beta+{\beta}^2) a_0 $+\\ $({\alpha}^3+2\alpha\beta)a_1$};
		\node (a01) at (0.5,-3)  [ align=center] {$({\alpha}^3 \beta+2\alpha{\beta}^2) a_0 $ +\\ $ ({\alpha}^4+3\alpha^2\beta+{\beta}^2)a_1$};
		\node (a02) at (1.5,-3)  [ align=center] {$({\alpha}^4\beta + 3\alpha^2{\beta}^2+{\beta}^3) a_0 $ +\\ $ ({\alpha}^5+ 4\alpha^3\beta+ 3\alpha{\beta}^2)a_1$};
		
		\end{tikzpicture}}
	\caption{\BT$(a_0,a_1,\alpha,\beta;\alpha,\beta)$}
	\label{fig:Triangle_alpha_beta}
\end{figure}

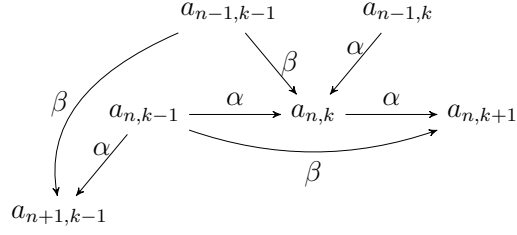
\begin{figure}[!thbp]
	\centering 
	\scalebox{0.90}{ 		         
		\begin{tikzpicture}[->,xscale=2.5,yscale=1.5, auto,swap]

		\node (a0) at (-0.5,1)   {$a_{n-1,k-1} $}; 
		\node (a00) at (0.5,1)   {$a_{n-1,k} $}; 		
		\node (a1) at (-1,0)   {$a_{n,k-1} $}; 
		\node (a2) at (-1.5,-1)   {$a_{n+1,k-1} $}; 
		\node (b) at (0,0)   {{$a_{n,k} $}};
		\node (c) at (1,0)   {{$a_{n,k+1} $}};		
		
		\path (a1) edge node [above]{$\alpha$} (a2);
		\path (a0) edge node [right]{$\beta$} (b);
		\path (a1) edge node [above]{$\alpha$} (b);
        \path (a0) edge [bend right] node [left]{$\beta$} (a2);
		\path (a00) edge node [above]{$\alpha$} (b);
		\path (b) edge node [above]{$\alpha$} (c);	
        \path (a1) edge [bend right] node [below]{$\beta$} (c);		
		\end{tikzpicture}}
	\caption{Coefficients of summing in case 1}
	\label{fig:same}
\end{figure}

\begin{theorem} \label{th:ank_binom}
The  entries in  triangle \BT$(a_0,a_1,\alpha,\beta;\alpha,\beta)$ can be written ($n\geq 1$) by
\begin{equation*}
a_{n,k} = a_0 \sum_{i=0}^{\lfloor \frac{n+k-2}{2}\rfloor} \binom{n+k-i-2}{i} \alpha^{n+k-2i-2}\beta^{i+1} + a_1 \sum_{i=0}^{\lfloor \frac{n+k-1}{2} \rfloor}\binom{n+k-i-1}{i} \alpha^{n+k-2i-1}\beta^{i}. 
\end{equation*}
\end{theorem}
\begin{proof}
We consider the results of Theorem~\ref{th:explicit}. Indeed, using the equalities $u=\alpha$, $v=\beta$ and $x_1+x_2=\alpha$, $\alpha x_1+\beta=x_1^2$, $\alpha x_2+\beta=x_2^2$, $D+\alpha=2x_1$, $D-\alpha=-2x_2$ the equations \eqref{eq:explicit}--\eqref{eq:exp_an1} become 
\begin{equation*}
a_{n,k}= \frac{-a_{n,0}x_2+a_{n,1}}{D}x_1^k + \frac{a_{n,0}x_1- a_{n,1}}{D}x_2^k,
\end{equation*}
where  $D=\sqrt{\alpha^2+4\beta}=x_1-x_2\ne 0$  and
\begin{eqnarray*}
a_{n,0} &=& \frac{-a_{0}x_2+a_{1}}{D}x_1^n + \frac{a_{0}x_1-a_{1}}{D}x_2^n,\\
a_{n,1} &=& \frac{-a_{0}x_2+a_{1}}{D}x_1^{n+1} + \frac{a_{0}x_1-a_1}{D}x_2^{n+1}.
\end{eqnarray*}
A little calculation shows that 
\begin{equation}\label{eq:Gir_0}
a_{n,k}=a_0\beta\left(\frac{x_1^{n+k-1}-x_2^{n+k-1}}{x_1-x_2}\right)
+a_1\left(\frac{x_1^{n+k}-x_2^{n+k}}{x_1-x_2}\right).
\end{equation}
Using the Girard-Waring formula \cite{G}
\begin{equation*}
\frac{X^{N+1}-Y^{N+1}}{X-Y} = \sum_{i=0}^{\lfloor\frac{N}{2}\rfloor} (-1)^i \binom{N-i}{i} (X+Y)^{N-2i}(XY)^i,
\end{equation*}
where, in our case, $X=x_1$, $Y=x_2$, $X+Y=\alpha$, $XY=-\beta$ and $N=n+k-2$ or $N=n+k-1$, the thesis follows. The Girard-Waring formula also holds in the case $x_1=x_2$, i.e., $D=0$, taking the limit $x_1 \rightarrow x_2$ on both members. We mention that $D=0$ implies $\mathcal{A}\ne0$.
\end{proof}

\begin{corollary}
	In case of the triangle \BT$(a_0,a_1,\alpha,\beta;\alpha,\beta)$ the  rising diagonal sequence $\{a_{n-k,k}\}_{k=0}^{\lfloor\frac{n}{2}\rfloor}$ is a constant sequence for any $k$.
\end{corollary}
\begin{proof} 
	We apply the formula proved in Theorem~\ref{th:ank_binom} to $a_{n-k,k}$ and obtain an expression that does not depend on the index $k$.
\end{proof}

Figure~\ref{fig:Triangle_case1} shows \BT$(a_0,a_1,1,1;1,1)$ as an example for case~1. Using the result of Theorem~\ref{th:ank_binom}, in the case $\alpha=1$ and $\beta=1$, the binomial coefficients are the coefficients of the elements in the rising diagonals.   Thus  $a_{n,k}=a_0 f_{n+k-1} + a_1 f_{n+k}$, where $f_n$ is the $n^{\text{th}}$ Fibonacci number ({A000045}), so that, $f_0=0, f_1=f_{-1}=1$. (We also find from the above relation \eqref{eq:Gir_0} the connection with Fibonacci numbers.) The special Fibonacci binomial interpolated triangle is \BT$(1,1,1,1;1,1)$ ({A199512}). Falcon and Plaza \cite{Falcon} provided an other example in Table~4 for this case, namely \BT$(0,1,3,1;3,1)$, which is generated by the $3$-Fibonacci sequence ({A006190}).

\begin{figure}[!h]
	\centering
	\scalebox{0.99}{ \begin{tikzpicture}[->,xscale=2.5,yscale=0.8, auto,swap]
		\node(a00) at (0,0)    {$a_0$};
		
		\node (a01) at (-0.5,-1)   {$a_1$};
		\node (a02) at (0.5,-1)   {$p+a_1$};
		
		\node (a01) at (-1,-2)   {$a_0+a_1$};
		\node (a02) at (0,-2)   {$a_0+2a_1$};	
		\node (a01) at (1,-2)   {$2a_0+3a_1$};
		
		\node (a02) at (-1.5,-3)   {$a_0+2a_1$};		
		\node (a01) at (-0.5,-3)   {$2a_0+3a_1$};
		\node (a01) at (0.5,-3)   {$3a_0+5a_1$};
		\node (a02) at (1.5,-3)   {$5a_0+8a_1$};	
		
		\node (a01) at (-2,-4)   {$2a_0+3a_1$};
		\node (a01) at (-1,-4)   {$3a_0+5a_1$};
		\node (a01) at (0,-4)   {$5a_0+8a_1$};
		\node (a02) at (1,-4)   {$8a_0+13a_1$};
		\node (a02) at (2,-4)   {$13a_0+21a_1$};
		
		\end{tikzpicture}}
	\caption{\BT$(a_0,a_1,1,1;1,1)$}
	\label{fig:Triangle_case1}
\end{figure}

\subsection{Case 2}

 From the equations system we obtain
\begin{alignat*}{2}
	v_1&=\frac{\alpha-1+\sqrt{(\alpha-1)^2+4\beta}}{2}, \quad  \quad u_1&=\frac{v_1^2}{\beta},\\
	v_2&=\frac{\alpha-1-\sqrt{(\alpha-1)^2+4\beta}}{2}, \quad \quad u_2&=\frac{v_2^2}{\beta}.	
\end{alignat*}

If $(\alpha-1)^2+4\beta=0$, i.e., $\beta=-(1/4)(\alpha-1)^2$, then, replacing the corresponding values for $u$ and $v$, after simplification we have
\begin{equation}\label{tri:case2_0}
\text{\BT}\left(a_0,a_1,\alpha,-\frac{(\alpha-1)^2}{4};-1,\frac{\alpha-1}{2}\right).
\end{equation}

If $(\alpha-1)^2+4\beta\ne0$, then the triangles are 
\begin{eqnarray}\label{tri:case2a}
	& &\text{\BT}(a_0,a_1,\alpha,\beta;u_1,v_1),\label{tri:case2_1}\\
	& &\text{\BT}(a_0,a_1,\alpha,\beta;u_2,v_2).\label{tri:case2_2}
\end{eqnarray}

Finally, we compare the triangles \eqref{tri:case1}--\eqref{tri:case2_2} with the symmetrical ones \eqref{tri:sym}--\eqref{tri:sym3}. In all the cases for the symmetrical binary interpolated binomial triangle with some calculations  we  gain the following corollary.

\begin{corollary} The symmetrical binary interpolated binomial triangles whose   an inner entry  could be the  sum of values not only left but also directly above it by coefficients $u$ and $v$ are
\begin{equation*}\label{tri:sym_case1}
\text{\BT}(a_0,-\frac{a_0}{2},-1,-1;-1,-1),
\end{equation*}
which has only terms $a_0$ and $\pm a_0/2$ and 
\begin{equation*}\label{tri:sym_case2}
\text{\BT}(a_0,a_0,2,-1;2,-1),
\end{equation*}
which has only terms $a_0$.
\end{corollary}

\begin{remark}
When both parameters $a_{0,0}$, $a_{1,0}$ are zero we find the trivial triangle exclusively composed of null entries.\\
The triangle with all entries equal to 1 correspond to  more binomial interpolated triangles, for example, $\text{\BT}(1,1,2,-1;-1,2)$ or $\text{\BT}(1,1,2,-1;-1,2)$.\\
The triangle whose left diagonal and rows are the sequences of natural numbers ({A000027}) is  $\text{\BT}(1,2,2,-1;2,-1)$ ({A094727}).
\end{remark}

\section{Acknowledgement}

The author would like to thank the anonymous referee of Journal of Integer Sequences for carefully reading the manuscript and for his/her useful suggestions and improvements.

\end{document}